\documentclass[12pt,a4paper]{article}
\title{Spin geometry of the
rational noncommutative torus}
\author{~\\
Alessandro Carotenuto and
Ludwik D\k abrowski\\[10pt]
\textit{\small Scuola Internazionale Superiore di Studi Avanzati (SISSA)}\\{\small  via Bonomea 265, I-34136 Trieste, Italy}}
\usepackage[latin1]{inputenc}
\usepackage{amsmath}
\usepackage{amsfonts}
\usepackage{amssymb}
\usepackage{amsthm}
\usepackage{bbm}
\usepackage{color}
\usepackage{slashed}
\usepackage[all]{xy}
\usepackage{stmaryrd}
\usepackage{ogonek}

\makeindex
\begin{document}
\maketitle
%\tableofcontents
\newtheorem{ev}{Everything}[section]
\newtheorem{exa}[ev]{Example}
\newtheorem{mydef}[ev]{Definition}
\newtheorem{theorem}[ev]{Theorem}
\newtheorem{lemma}[ev]{Lemma}
\newtheorem{prop}[ev]{Proposition}

\def\C{{\mathbb C}}
\def\N{{\mathbb N}}
\def\R{{\mathbb R}}
\def\Z{{\mathbb Z}}
\def\z2{{\mathbb Z}/2}
\def\T{{\mathbb T}}
\newcommand{\D}{\slashed{D}}
%\newcommand{\mapsfrom}{\mathrel{\reflectbox{\ensuremath{\map‌​sto}}}}

%\usepackage{hyperref}
%\usepackage{utopia}
%\usepackage[utopia]{mathdesign}

%LD
\newcommand{\ver}[1]{{\color{green}#1}} %TMP
\newcommand{\blu}[1]{{\color{blue}#1}} %TMP
\newcommand{\red}[1]{{\color{red}#1}} %TMP
\newcommand{\segna}[1]{\rosso{$\overline{\rule{0pt}{1pt}\smash[t]{\text{\color{black}#1}}}$}}
\newcommand{\half}{\frac{1}{2}}

\newcommand{\Tb}{\mathbb{T}}         %% torus
\newcommand{\bZ}{\mathbb{Z}}         %% integers
\newcommand{\bN}{\mathbb{N}}         %% natural numbers
\newcommand{\bR}{\mathbb{R}}         %% real numbers
\newcommand{\bC}{\mathbb{C}}         %% complex numbers
\newcommand{\te}{\theta}
\newcommand{\cO}{\mathcal{O}}        %% an algebra
\newcommand{\Dslash}{{D\mkern-11.5mu/\,}} %% Dirac operator
\def\vol{{\mathrm{vol}}}

%=============================

\begin{abstract}\noindent
The twined almost commutative structure of
the standard spectral triple on the noncommutative torus with rational parameter is exhibited, by showing isomorphisms with a spectral triple on the algebra of sections of certain bundle of algebras, and a spectral triple on a certain invariant subalgebra of the product algebra.
These isomorphisms intertwine also the grading and real structure.
This holds for all four inequivalent spin structures,
which are explicitly constructed in terms of double coverings of the noncommutative torus (with arbitrary real parameter).
These results are extended also to a class of curved (non flat) spectral triples,
obtained as a perturbation of the standard one by eight  central elements.
\end{abstract}

\section{Introduction}
The framework of Connes' noncommutative spectral geometry provides a generalisation of ordinary
Riemannian spin manifolds to noncommutative manifolds.
Within this framework, the special case
of a (globally trivial) almost-commutative manifold has been shown to describe a (classical) gauge
theory over a Riemannian spin manifold (see e.g. \cite{vS}), which ultimately led to a description of the full Standard
Model of high energy physics, including the Higgs mechanism and neutrino mixing (\cite{vS},\cite{C1},\cite{C2},\cite{C4},\cite{CM},\cite{DLM}).
The gauge theories mentioned above are, by construction, topologically trivial, in the sense that
the corresponding principal bundles are globally trivial bundles.
A generalization to nontrivial bundles has
been obtained in (\cite{koen}) for the special case of Yang-Mills theory.

In this paper we approach the nontriviality issues
in another way,
 without employing connections on principal bundles.
We deal with a {\em twined almost commutative geometry}
directly on the Riemannian spin level of spectral geometry.
The classical situation we have in mind is
the quotient $(M\times N)/G$ of the product of two compact spin $G$-manifolds $M$ and $N$ by the diagonal action of a Lie group $G$, so that it forms a
locally trivial bundle over
$M/G$ with a fiber $N$.
More specifically we think of coverings of manifolds
and in particular finite coverings.
Of course in all that it should be assumed that the bundle projection is a Riemannian submersion.

As it is customary in noncommutative geometry,
we work with the dual formulation, and restrict to the case when $M$ is classical while $N$
is a quantum space described by a
star algebra $A$. Namely, we build a spectral triple
from the canonical spectral triple
on a spin $G$-manifold $M$ and a noncommutative spectral triple
that describes the geometry of a quantum $G$-space $N$.
We call this construct {\em twined} when  is not globally trivial, i.e. not a product one.
We prefer the adjective {\em twined} rather then
{\em twisted}, that is often used for spectral triples to mean twisted commutators, twisted real structure, or a coupling to certain differential 3-form.

A particular case we focus on is a twined almost commutative spectral triple, when the noncommutative one is finite.
Then a suitable subalgebra of the
product of the algebras of the two spectral triples above can be regarded as the algebra of sections of a locally trivial
finite dimensional algebra bundle over $M/G$.
This is known to happen for example for the noncommutative torus with rational parameter (\cite{bondia}),
for which we provide in Subsection \ref{First isomorphic spectral triple} all  the other data of the spectral triple,
namely the Hilbert space and the analogue of the Dirac operator,
as well as grading and real structure.

We elaborate also in Subsection \ref{Second isomorphic spectral triple} a closely related third description in which the algebra is realized as a suitable subalgebra of the tensor product of two algebras.
It is suited
also for infinite dimensional algebras and overcomes the requirement that in the second description vector bundles need to be of finite rank.
This will allow us to consider in a future work also e.g. the theta deformation of the canonical spectral triple \cite{CDV} (see also \cite{CL}), as the twined product of the canonical spectral triple on the quotient manifold by a torus of isometries, with the quantum torus.

In Section \ref{Inequivalent spin structures} we describe the four inequivalent spin structures on the noncommutative torus
with arbitrary real value of the deformation parameter $\theta$ in terms of inequivalent double coverings.
We construct then the corresponding standard (flat) spectral triples, which are mutually non-isomorphic already on the level of algebras.
They should be compared with those constructed in \cite{PS} in terms of spectral triples, but
the precise correspondence may be quite involved due to rather different classification criteria of inequivalent spin structures.
In addition note that in contrast with \cite{PS}, though following the same scheme, it is claimed in \cite{V13} that for rational parameter $\theta$ the number of inequivalent spin structures can be less than four.

In subsection \ref{Isomorphic spectral triples} we indeed
analyse the case of rational parameter $\theta$ and provide two isomorphic descriptions of spectral triples for all four spin structures in analogy to the case of trivial spin structure presented in Subsections
\ref{First isomorphic spectral triple} and \ref{Second isomorphic spectral triple}.

Finally, in Section \ref{Curved rational noncommutative torus}
we discuss the three isomorphic descriptions of spectral triples on the curved
(i.e.\,non flat) rational noncommutative torus. After presenting two simple examples of perturbations of the standard Dirac operator we then show that a wide class
of perturbations by means of elements of the algebra preserves the twined almost commutative structure if and only if these elements are central.

It will be interesting to investigate in future works possible relations of our spectral geometry approach with the approach of \cite{koen} which make use of connections on principal bundles,
and also with \cite{bcr16}, where spectral triples on the noncommutative torus have been employed to study topological insulators.
Another direction to be explored regards the noncommutative coverings
and in particular spectral triples on self-coverings of the rational
noncommutative torus in \cite{agi}, and the quantum coverings
studied in \cite{Canlubo} via Hopf algebroids as the noncommutative analogue of groupoids.

\section{Preliminaries}
%=================\\

Geometry of a compact Riemannian spin manifold $M$ can be encoded (\cite{C1},\cite{bondia}) in terms of its canonical spectral triple, which consists of the algebra of smooth complex functions on $M$, the Hilbert space $L^2(M,\Sigma)$ of square integrable Dirac spinors on $M$ and the Dirac operator on $M$. More generally, {\em spectral triple} is the datum $(A_,\mathcal{H}, D)$ of a unital $*$-algebra $A$, a Hilbert space $\mathcal{H}$ carrying a faithful unitary representation $\pi:A\to B(\mathcal{H})$, and a selfadjoint operator $D$ on $\mathcal{H}$ with compact resolvent, such that the commutators $[D,\pi(a)]$ are bounded operators for any $a\in A$. \\
A spectral triple is called {\em even} if there is a $\Z/2$-grading operator $\chi$ commuting with $\pi(a)$ for any $a\in A$ and anticommuting with $D.$
Furthermore it is called {\em real} if there is a
$\C$-antiunitary operator $J,$ such that
$[a,JbJ]=0$ for $a,b\in A$, $J^2=\epsilon $,
$JD=\epsilon'DJ$ and $J\chi=\epsilon''\chi J$.
The  three signs $\epsilon, \epsilon', \epsilon''$ determine the so called KO-dimension of the spectral triple.

\begin{mydef}\label{diags}
We call two spectral triples $(A_1,\mathcal{H}_{1}, D_1)$ and $(A_2,\mathcal{H}_2, D_2)$ isomorphic iff there exist an isomorphism of algebras
$\gamma :A_1\to A_2$ together with unitary operator $T:H_1\to H_2$ such that
%the following two diagrams commute
\begin{equation}\label{diag1}
T\pi_{1}= \pi_{2} (\gamma\times T)
\end{equation}
%\begin{displaymath}
and
\begin{equation}\label{diagD}
TD_1=D_2\,T.
  \end{equation}
In case of spectral triples that are even, real or both,
we require in addition that
\begin{equation}\label{diagchi}
T\chi_1=\chi_2\,T,
  \end{equation}
\begin{equation}\label{diagJ}
TJ_1=J_2\,T,
  \end{equation}
or both.
\end{mydef}
\noindent
Note that \eqref{diag1} can be equivalently stated in terms of  $\hat\pi:A\to B(\mathcal{H})$, $\hat\pi(a)h=\pi(a,h)$,
as
%commutativity of
\begin{equation}\label{diag2}
Ad_T\, \hat\pi_{1} = \hat\pi_{2}\,\gamma .
\end{equation}

%\noindent
We remark that when two even real spectral triples are isomorphic
then they must have the same KO-dimension, e.g.
\eqref{diagD} and \eqref{diagJ} require that the following diagram commutes
\begin{equation}
    \xymatrix{  \mathcal{H}_{1} \ar[rr]^{J_1D_{1}} \ar[ddd] \ar[dr]^{D_1J_1} & &  \mathcal{H}_{1} \ar[ddd]  \\
   & \mathcal{H}_{1} \ar[d] \ar[ur]^{\epsilon'} & \\
    & \mathcal{H}_{2} \ar[dr]^{\epsilon'}  &\\
    \mathcal{H}_{2}\ar[rr]^{J_2D_2} \ar[ur]^{D_2J_2} &  & \mathcal{H}_{2}\,.}\\
  \end{equation}
Recall that {\em product} of two (even) spectral triples $(A_1,\mathcal{H}_1, D_1)$ with $(A_2,\mathcal{H}_2, D_2)$, that corresponds to the usual direct product of two (even dimensional) Riemannian spin manifolds, is given by
\begin{equation}\label{product}
(A_1\otimes A_2, \mathcal{H}_{1}\otimes \mathcal{H}_2,
D_1\otimes 1 +\chi_1\otimes D_2),
\end{equation}
where we assume that the algebraic tensor product of algebras can be suitably completed, and we use the usual tensor product of Hilbert spaces. Equivalently one can take
$D_1\otimes \chi_2+1\otimes D_2$ as Dirac operator.
Furthermore in case
of even real spectral triples the grading is $\chi_1\otimes \chi_2$, and the real structure is $J_1\otimes J_2$.

By {\em almost commutative spectral triple} we understand the product of the canonical spectral triple
$(C^\infty(M), L^2(M, \Sigma), \D)$
with a finite spectral triple, i.e. one with finite dimensional Hilbert space $\mathcal{H}_2=\C^n$.
Thus, taking advantage of the isomorphism $L^2(M, \Sigma)\otimes  \C^n\approx L^2(M, \Sigma\times (M\times\C^n))$ we see that the Hilbert space consists  of $n$-copies of Dirac spinors (globally), and the algebra of smooth $A_2$-valued functions on $M$.
Now we introduce a generalization of this notion, in which we allow the Hilbert space to consist of $L^2$ sections of the product of $\Sigma$ with a (locally trivial) vector bundle with a typical fiber $\C^n$, and correspondingly the algebra to consist of smooth sections of some (locally trivial) bundle of finite  dimensional 
$*$-algebras.\\

\begin{mydef}\label{def}
{\em Twined almost commutative spectral triple} is
a spectral triple of the form
$(C^\infty(M,F), L^2(M,\Sigma\otimes E), D)$, where
$F$ is an algebra sub-bundle of endomorphisms of a
(locally trivial) finite rank hermitian vector bundle $E$ on $M$
and the operator $D$ has locally an almost product form
\begin{equation}\label{almprod}
D=\D\otimes {\rm id}_E + \chi\otimes D_E,
\end{equation}
where $D_E\in {\rm End(E)}$.
\end{mydef}

The adjective {\em twined} refers here to a possibility that $F$ is globally nontrivial as a vector bundle or as a bundle of algebras. We concede that
\eqref{almprod} will be often written as the composition of operators
$$D=\D\,{\rm id}_E + \chi D_E .$$

A special case of the above definition occurs when $M$ is a
quotient of another Riemannian spin manifold $\tilde{M}$
by a group of isometries (for simplicity assumed to be discrete, or even finite, and preserving the spin structure).
Then our spectral triples can be built from the canonical spectral triple
$(C^\infty(\tilde{M}), L^2(\Sigma), \D)$
on $\tilde{M}$ and a finite $G$-equivariant noncommutative spectral triple $(A, H, D)$.
More precisely, the relevant algebra will be given by the $G$-invariant subalgebra of $C^\infty(\tilde{M})\otimes A$,
the Hilbert space given by the $G$-invariant Hilbert subspace of $L^2(\Sigma)\otimes H$ and the Dirac-type operator given by $\D  \otimes 1+ \chi \otimes D $,
where $\chi$ is the chiral grading of $L^2(\Sigma)$.
This setup lends itself however to a generalization to
not necessarily finite spectral triples (i.e. to infinite dimensional $H$)
and to full-fledged nontrivial bundles $F$, which would require the use of Hilbert modules rather that sections of finite rank vector bundles as in Definition \ref{def}, and possibly internal products of spectral triples (see e.g. \cite{M}).

\section{Rational noncommutative torus}

In this section we establish an isomorphism between the standard spectral triple on the rational noncommutative torus, that is  $\left( C^\infty(\mathbb{T}^2_{p/q}), L^2(\mathbb{T}^2_{p/q})\otimes \mathbb{C}^2, D_{p/q}\right),$ and two other spectral triples.  The first one is $\left(\Gamma^\infty(F),  L^2(F)\otimes \mathbb{C}^2, D_F\right)$, where $F$ is
an algebra bundle of $q\times q$ matrices over the torus $\mathbb{T}^2$, and $D_F$ is certain differential operator on $F$.
This construction originates from a known (see e.g. \cite{bondia}) isomorphism between the algebras $C(\mathbb{T}^2_{p/q})$ and the continuous sections
of $F$, which as a bundle of algebras is not a product bundle.

It is an example of Definition \ref{def}
with $F$ regarded as self-endomorphisms consisting of fiber-wise left multiplication.
The second one, denoted as $\left(\mathcal{A}_{p/q},\mathcal{H}_{p/q},\mathcal{D}_{p/q}\right),$
can be interpreted as a twined almost commutative spectral triple as well since
$\mathcal{A}_{p/q}$ is the subalgebra of $\mathbb{Z}_q \times \mathbb{Z}_q$-invariants of $C(\mathbb{T}^2)\otimes M_q$,
 $\mathcal{H}_{p/q}$ is a Hilbert subspace of $\mathbb{Z}_q \otimes \mathbb{Z}_q$-invariants in $L^2(\mathbb{T}^2)\otimes M_q\otimes \mathbb{C}^2$,
and $D_{p/q}$ is the canonical Dirac operator $\D$ on $\mathbb{T}^2$ tensor the identity on $M_q$.

The Hilbert spaces of both these spectral triples are given as completions of the corresponding two algebras with respect to certain norms, while the Dirac operators are defined in such a way that their action on the respective spaces satisfy \eqref{diagD}.\\

\subsection{The standard spectral triple}\label{The standard spectral triple}

We recall the definition of the standard spectral triple
on the noncommutative torus $\mathbb{T}_\theta$
with a parameter $\theta$.\\

\noindent
$\bullet$ The algebra.
\begin{mydef}
Let $U,V$ be two unitary generators with the commutation relation
\begin{equation}
UV=\lambda VU,
\end{equation}
where $\lambda=e^{2 \pi i \theta}$, $0\leq\theta\leq 1$. The algebra $C^\infty(\mathbb{T}^2_\theta)$ of smooth complex valued functions on the noncommutative torus consists of the series:
\begin{equation}\label{ncf}
a=\sum_{(m,n)\in \mathbb{Z}_2}a_{mn}U^{m}V^{n},
\end{equation}
where
the double sequence of $a_{mn}\in \C$ satisfies
\begin{equation}\label{bdda}
\parallel a \parallel_k:= \sup_{(m,n) \in \mathbb{Z}^2}(1+m^2+n^2)^k \mid a_{mn}\mid^2< \infty, \quad \forall k\in \N.
\end{equation}
\end{mydef}
Clearly, when $ \theta=0$ the algebra $C^\infty(\mathbb{T}^2_0)$ is isomorphic to the algebra
$C^\infty\left(\mathbb{T}^2\right)$ of smooth complex functions on the classical torus
$\mathbb{T}^2:=\{(z_1,z_2)\in\C^2\,|\,|z_1|=1=|z_2|\}$,
generated by the identity functions on the two factors $S^1\subset \C$ denoted (with a slight abuse of notation) by $z_1$ and $z_2$ and called coordinate functions on $\mathbb{T}^2$. It should be clear from the context it we regard $z_1$ and $z_2$ as numbers or as functions.\\

\noindent$\bullet$ The Hilbert space $L^2\left(\mathbb{T}^2_{\theta}\right)\otimes\C^2$.\\
Denote by $\mathfrak{t}$ the following tracial state on
$C^\infty\left(\mathbb{T}^2_{\theta}\right):$
\begin{equation}\label{Trtoro}
\mathfrak{t}\left(\sum a_{mn}U^{m}V^n\right)=a_{00}\ ,
\end{equation}
where $a_{00}$ is the coefficient of $1.$ We will refer to $\mathfrak{t}$ as {\em trace}.
\\The trace defines a sesquilinear form on $C^\infty\left(\mathbb{T}^2_{\theta}\right)$ by
\begin{equation}
\langle a\mid b \rangle= \mathfrak{t}(a^*b)
\end{equation}
and a norm:
\begin{equation}
\mid \mid a \mid \mid = \pm \sqrt{\mathfrak{t}(a^*a)}.
\end{equation}
We denote $L^2\left(\mathbb{T}^2_{\theta}\right)$ the Hilbert space obtained by completion of $C^\infty\left(\mathbb{T}^2_{\theta}\right)$ with respect to this norm.
It carries a $*-$representation of $C^\infty\left(\mathbb{T}^2_{\theta}\right)$ by left multiplication:
\begin{equation}
\pi(a): b\mapsto ab,
\end{equation}
i.e. $L^2\left(\mathbb{T}^2_{\theta}\right)$ is a left $*$-module over $C^\infty\left(\mathbb{T}^2_{\theta}\right).$
The elements of $L^2\left(\mathbb{T}^2_{\theta}\right)$
are analogues of Weyl spinors on the noncommutative torus.

As the full Hilbert space of analogues of Dirac spinors on the noncommutative torus we take $L^2\left(\mathbb{T}^2_{\theta}\right)\otimes\C^2$,
with the diagonal $*$-module structure over $C^\infty\left(\mathbb{T}^2_{\theta}\right).$\\

\noindent$\bullet$ The Dirac operator $D_{\theta}$.\\
The trace \eqref{Trtoro} is invariant under the actions of the  torus group $\mathbb{T}^2$ on $C^{\infty}\left(\mathbb{T}^2_{\theta}\right)$ by automorphisms defined by
\begin{equation}\label{U1U1}
U\mapsto z_1 U, \quad V\mapsto z_2 V,\quad \forall  (z_1,z_2)\in \mathbb{T}^2.
\end{equation}

   These actions are infinitesimally generated by the two commutating derivations
\begin{eqnarray}
\delta_{1}U= i U, \quad \delta_{1}V=0
\\  \delta_{2}U=0, \quad \delta_{2}V= i V.
\end{eqnarray}
The canonical flat Dirac operator on the Hilbert space $L^2(\mathbb{T}^2_{\theta})\otimes\C^2$ is a
contraction of derivations $\delta_\ell$  with Pauli matrices $\sigma_\ell$ (Clifford multiplication):
\begin{equation}\label{Dirac}
D_{\theta}=i\left(\sigma_1\delta_1 +\sigma_2 \delta_2 \right):=
\left(\begin{matrix} 0 && i\delta_1+\delta_2
\\  i\delta_1-\delta_2 && 0
\end{matrix}\right).
\end{equation}
Recall that the spectral triple $\left( C^\infty\left(\mathbb{T}^2_{\theta}\right), L^2\left(\mathbb{T}^2_{\theta}\right)\otimes \mathbb{C}^2, D_{\theta}\right)$ is even, with the grading $\chi_{\theta}=\left(\begin{matrix} 1 & 0
\\0 &-1
\end{matrix}\right)$ that commutes with every $a \in C^\infty\left(\mathbb{T}^2_{\theta}\right)$ and anticommutes with $D_{\theta}.$
\\This spectral triple is also real, by taking as real structure
\begin{equation}\label{rs}
  J_{\theta}=-i J^0_{\theta} \otimes (\sigma_2 \circ c.c.),
  \end{equation}
   where $J_0: \mathcal{H}^0_{\theta}\rightarrow \mathcal{H}^0_{\theta}$ is the Tomita conjugation:
  \begin{equation}
  J^0(a)=a^*.
  \end{equation}

It is immediately seen that for $\theta =0 $
the derivations $\delta_\ell$ become the coordinate derivatives
that can be expressed also as $\partial_\ell=z_\ell \frac{\partial}{\partial z_\ell}$. Furthermore
%on the usual torus $\mathbb{T}^2$
 the standard spectral triple described above is just the canonical spectral triple and in particular
\begin{equation}\label{D0}
D_{0}=i\left(\sigma_1\partial_1+\sigma_2\partial_2 \right)=
 \left(\begin{matrix} 0 && i\partial_1+\partial_2
\\  i\partial_1-\partial_2 && 0
\end{matrix}\right)
\end{equation}
is the Dirac operator constructed from the (flat) Levi-Civita connection.
It should be however mentioned that this
corresponds to a particular choice of a spin structure on the noncommutative torus; we will describe the other spin structures in Section\,\ref{Inequivalent spin structures}.

In the following two subsections we will focus on the case in when $\theta$ is a rational number, so unless stated differently from now on:
\begin{equation}
\theta=p/q, \quad\mathrm{i.e.,}\quad
\lambda= e^{2 \pi i p/q},
\end{equation}
where $0<p<q\in \Z$ are relatively prime.
In this case the center $\mathcal{Z}_{p/q}$
of $C^\infty\left(\mathbb{T}^2_{p/q}\right)$
is generated by $U^q$ and $V^q$, and is just the invariant subalgebra for the finite subgroup $G\approx \mathbb{Z}_q \times \mathbb{Z}_q$
of \eqref{U1U1} of pairs of $q$th roots of 1.
%isomorphic to  $G=\mathbb{Z}_q \times \mathbb{Z}_q$.
The center $\mathcal{Z}_{p/q}$ is isomorphic to
$C^\infty\left(\check{\mathbb{T}}^2\right)$, where
$\check{\mathbb{T}}^2$ is the quotient of $\mathbb{T}^2$
by the free action $\kappa$ of $G$
%=\mathbb{Z}_q \times \mathbb{Z}_q$,
given by
\begin{equation}\label{kappa}
\kappa_{m,n}(z_1,z_2)=(\lambda^m z_1,\lambda^nz_2).
\end{equation}
Clearly $\mathbb{T}^2$ is a $q^2$-fold covering of $\check{\mathbb{T}}^2$ (a principal $G$-bundle),
but $\check{\mathbb{T}}^2$ is also diffeomorphic to a torus.
We denote by $[z_1,z_2]_\kappa$ the $\kappa$-equivalence classes (orbits of $\kappa$).
From the metric point of view we will equip $\mathbb{T}^2$ first with the standard flat Riemannian metric,
and then also with some other $G$-invariant ones. They descend to $\check{\mathbb{T}}^2$ so that $\pi$ is an isometric submersion. Then $\check{\mathbb{T}}^2$ is actually isometric to  ${\mathbb{T}}^2$
when the latter one is equipped with the original metric rescaled by $q^2$.
These metric properties reflect themselves via certain invariance properties of $D_0$ in expression \eqref{Dirac}. Namely it commutes with the derivations $\delta_\ell$ and with the torus group
action \eqref{U1U1}
%$U(1)\times U(1)$ of automorphisms
they generate; thus in particular it is invariant under the subgroup $G$.

\subsection{First isomorphic spectral triple}\label{First isomorphic spectral triple}

\noindent$\bullet$ The algebra $\Gamma^\infty(F)$.

As it is well known the C*-algebra $C\left(\mathbb{T}^2_{p/q}\right)$ of the rational noncommutative torus is isomorphic to the algebra of continuous sections of certain vector bundle $F$ of $q\times q$ matrix algebras, over a 2-torus.
The same holds of course also on the smooth level,
as we will present now in full detail.
%For that we need some notation.

Denote by $M_q$  the algebra of $q \times q$ complex matrices, and define $R,S\in M_q$ as
%the following two unitary matrices:
\begin{eqnarray}\label{mRS}
R=\left(\begin{matrix} 1 & & & 0
\\ &\lambda & &
\\& &\lambda^2  &
\\ ...&...&...&...
\\ & & & \lambda^{q-1}
\end{matrix}\right), & S=\left(\begin{matrix} 0 & & & 1
\\ 1&0 & &
\\& 1& &
\\ ...&...&...&...
\\ & &1 & 0
\end{matrix}\right).
\end{eqnarray}
We have $R^{q}=S^{q}=\mathbbm{1}$ and $RS=\lambda SR.$
Consider another action $\tau$ of $G$ on $\mathbb{T}^2\times M_q$:
% given by
$$\tau_{m,n}(z_1,z_2, A)=(\lambda^m z_1,\lambda^nz_2, R^mS^nAS^{-n}R^{-m}),\quad \forall A\in M_q\ .$$
%where $$.

We denote by $[z_1,z_2,A]_\tau $
the $\tau$-equivalence classes (orbits).
The space $F$ of orbits
 of $\tau$ forms a vector bundle over $\check{\mathbb{T}}^2$ with typical fiber $M_q$, and (well defined) projection
 $$\pi_F:F\to \check{\mathbb{T}}^2,\quad
 \pi_F:[z_1,z_2,A]_\tau\mapsto [z_1,z_2]_\kappa \ .$$
 % (which is well defined on orbits).\\

\noindent
We remark that the bundle $F$ is associated to the principal $G$-bundle ${\mathbb{T}^2}$ over
$\check{\mathbb{T}}^2$, via the representation $\rho:G\to End(M_q)$, given by
$$\rho_{m,n}( A)=R^mS^nAS^{-n}R^{-m},\quad A\in M_q\, . $$
Indeed, the assignment
$$F\ni[z_1,z_2,A]_\tau\mapsto [z_1,z_2,A]_\rho\in {\mathbb{T}^2}\times_\rho M_q\, ,$$
is well defined since $[\cdot, \cdot, \cdot]_\rho$ are the equivalence classes of the relation
$$(\lambda^m z_1,\lambda^n z_2,A)\sim (z_1,z_2,\rho_{m,n}^{-1}A),$$
and is an isomorphism. $\quad\diamond$\\

The smooth sections of $F$ form a $*$-algebra $\Gamma^\infty(F)$ with respect to the point-wise multiplication and point-wise hermitian conjugation (of matrices).
Its obvious completion is the C*-algebra of continuous sections.
With these observations we can state:
\begin{lemma}\label{ia1}
The map $Q$ defined on the generators by
\begin{equation}\label{map}
U\mapsto \xi_{U},\, V\mapsto\xi_{V},
\end{equation}
where
$$
\xi_U:\check{\mathbb{T}}^2\to F, \quad
[z_1,z_2]_\kappa\mapsto [z_1,z_2,z_1S]_\tau\ ,
$$
$$
\xi_V:\check{\mathbb{T}}^2\to F, \quad
[z_1,z_2]_\kappa\mapsto [z_1,z_2,z_2R^{-1}]_\tau
$$
extends to a $*$-isomorphism of algebras $Q:C^{\infty}\left(\mathbb{T}^2_{p/q}\right)\rightarrow
\Gamma^\infty(F)$ .
\end{lemma}
\begin{proof}
It is straightforward to check that $\xi_U$ and $\xi_V$ are well defined and \eqref{map} extends to a $*$-isomorphism due to the properties of the Fourier coefficients of a smooth function and the exchange rule
$\xi_{U} \xi_{V}= \lambda\, \xi_{V}\xi_{U}$.
\end{proof}
\noindent
%\red{
%\noindent
%Remark.

Note that $F$ is trivial as a vector bundle but
%when $\theta\neq 0$
it is nontrivial as the bundle of algebras.
Indeed the monomials ${\xi_{U}}^{m}{\xi_{V}}^n$
define a basis of $\Gamma^\infty(F)$ over $C^\infty(\check{\mathbb{T}}^2)$ and
any $\xi\in \Gamma^\infty(F)$ can be written as
\begin{equation}\label{basisF}
 \xi=\sum_{m,n=1}^q f_{mn}{\xi_{U}}^{m}{\xi_{V}}^n,
\end{equation}
where $f_{mn}\in C^\infty(\check{\mathbb{T}}^2),\; \forall\, 1\leq m,n\leq q$.
By viewing the coefficients $f_{mn}$ as a $q\times q$ matrix of functions in $C^\infty(\check{\mathbb{T}}^2)$ we can write an isomorphism of vector bundles
$$F\approx \check{\mathbb{T}}^2 \times M_q, \quad \xi([z_1,z_2]_\kappa)\mapsto
\left([z_1,z_2]_\kappa, f_{mn}([z_1,z_2]_\kappa\right).$$
However for $\theta=p/q\neq 0$ considered here this is not an isomorphism of algebra bundles
since the multiplication of sections $\xi$ does not correspond to the matrix multiplication of $f_{mn}$.

%\noindent
%Remark.
%Note also that 
As a matter of fact $F$ is the bundle of all vertical (or based)
endomorphisms of another complex vector bundle $E$ of rank $q$. Namely, $E$ is the orbit space of another free action of $G$ this time on ${\mathbb{T}}^2\times \C^q$, given by
$$(z_1,z_2, r)\mapsto(\lambda^m z_1,\lambda^nz_2, R^mS^nr),$$
where $r\in \C^q$. In fact
%it can be easily checked that
the (completed) algebras
$C\left(\mathbb{T}^2_{p/q}\right)$  and
$C\left(\check{\mathbb{T}}^2\right)$ are strong Morita equivalent via the ${C\!\left(\check{\mathbb{T}}^2\right)\!-\! C\!\left(\mathbb{T}^2_{p/q}\right)}$\, bimodule of
continuous sections of $E$.
%\,$\diamond$\\
The bundle $E$ won't play however any role in the definition of the Hilbert space representation of
$\Gamma^\infty(F)$, for which we shall employ the
bundle $F$ itself, with $\Gamma^\infty(F)$ self action by left multiplication (left regular representation).\\

%\noindent

Observe also that the center of $\Gamma^\infty(F)$ is generated by
$\xi_{U^q}=\xi_{U}^q$ and $\xi_{V^q}=\xi_{V}^q$ and is isomorphic to the center $\mathcal{Z}_{p/q}$
of $C^\infty\left(\mathbb{T}^2_{p/q}\right)$, and thus also to
$C^\infty\left(\check{\mathbb{T}}^2\right)$, in turn
identified with the $G$-invariant subalgebra
$C^\infty({\mathbb{T}}^2)^{G}$ of
$C^\infty({\mathbb{T}}^2)$,
via the map that sends $\xi_{U}^q\mapsto z_1^q$ and
$\xi_{V}^q\mapsto z_2^q$. Of course over this isomorphism $\Gamma^\infty(F)$ and
$C^\infty\left(\mathbb{T}^2_{p/q}\right)$ are isomorphic
as modules over their centers. \\

\noindent$\bullet$
The Hilbert space $L^2(F)\otimes \mathbb{C}^2$.

Now we look for a Hilbert space which can serve as a codomain
of the isometric extension of $Q$ to $L^2(\mathbb{T}^2_{p/q})$.
For that define on $\Gamma(F)$
a tracial state $\mathfrak{t}_F:=Q\mathfrak{t}Q^{-1} $, i.e.
\begin{equation}
\mathfrak{t}_F\left(\xi\right)
=\int_{\check{\mathbb{T}}^2}f_{00}
\end{equation}
for $\xi$ as in \eqref{basisF}, where $\int_{\check{\mathbb{T}}^2}$ is the normalized integral. The corresponding sesquiliniar form reads
\begin{equation}\label{sesqF}
\langle \xi \mid \xi' \rangle= \mathfrak{t}_F(\xi^{*}\xi'),
\end{equation}
where
\begin{equation}
\xi^*=\sum_{m,n=1}^q \bar f_{mn}{\xi_{V}}^{-n}{\xi_{U}}^{-m}.
\end{equation}

We define first the Hilbert space $L^2\left(F\right)$ as the completion of $\Gamma(F)$ with respect to the norm defined by the scalar product \eqref{sesqF}.
%$L^2\left(F\right)$
It carries a $*-$representation of $\Gamma(F)$ by left multiplication, i.e.\,it is a $\Gamma(F)-$module
(and similarly for $\Gamma^\infty(F)$). Then as the full Hilbert space we take
$L^2(F)\otimes \mathbb{C}^2$.
Taking advantage of the (inverse) isomorphism $Q$ given by \eqref{map}
and its Hilbert space amplification we have
\begin{lemma}\label{im1}
The $*$-representation of $\Gamma^\infty(F)$
on $L^2(F)\otimes \mathbb{C}^2$ is unitary equivalent
to the $*$-representation of $C^\infty\left(\mathbb{T}^2_{p/q}\right)$ on
 $ L^2\left(\mathbb{T}^2_{p/q}\right)\otimes \mathbb{C}^2$.
\end{lemma}

\noindent$\bullet$ The Dirac operator $D_F$.

 We define the derivations algebra $\Gamma^\infty(F)$ as
 $\partial^F_\ell:= Q \delta_\ell Q^{-1}$, so that their actions on the generators of $\Gamma^\infty(F)$ are just
\begin{equation}
%\begin{split}
\partial^F_1 \xi_U=  i \xi_U, \quad \partial^F_1 \xi_V=0, \quad
%\\
\partial^F_2 \xi_U= 0,\quad \partial^F_2 \xi_V=i \xi_V.
%\end{split}
\end{equation}
The Dirac operator $D_F$ which
%makes the diagram
satisfies \eqref{diagD} for $D_1=D_\theta$ and $D_2=D_F$ is then
%commute is
\begin{equation}\label{df}
D_F=i(\sigma_1\partial^F_{1}  + \sigma_2\partial^F_{2} )\ .
\end{equation}
% and extends as derivations to the whole algebra $\Gamma^\infty(F)$.

We call (locally defined) $M_q$-valued function $\widetilde\xi$
on $\T^2$ {\em local components} of $\xi$
when
$$\xi([z_1,z_2]_\kappa)= [z_1,z_2,\widetilde\xi(z_1,z_2)]_\tau.$$
In particular the local components of $\xi_U$ and $\xi_V$ are respectively 
$z_1S$ and $z_2R^{-1}$.
Next we call $\widetilde T$ local components of an operator $T$ on $\Gamma^\infty(F)$
when  $\widetilde{T\xi} = \widetilde T\, \widetilde\xi$,
and similarly for operators on $L^2(F)\otimes \mathbb{C}^2$.
In particular the local components of the differential operators $\partial^F_\ell$
are simply the coordinate derivatives
\begin{equation}\label{derFF}
\widetilde{\partial^F_\ell} = \partial_\ell\ .
%\widetilde\xi .
\end{equation}
Thus local components of $D_F$ are 
\begin{equation}\label{dff}
\widetilde{D_F} =i(\sigma_1\partial_{1}  + \sigma_2\partial_{2} ).
%\widetilde\xi .
\end{equation}
%\red{
Note that $\widetilde{D_F}$ looks quite like the canonical Dirac operator \eqref{D0}
on the torus constructed from the (flat) Levi-Civita connection of the standard metric on $\mathbb{T}^2$,
and $D_F$ in fact is unitary equivalent to \eqref{D0}.\\

\noindent$\bullet$ The isomorphism.

By using Lemmata \eqref{ia1} and \eqref{im1} and the above discussion, we obtain:
\begin{prop}
The spectral triple $\left(\Gamma^\infty(F),  L^2(F)\otimes \mathbb{C}^2, D_F\right)$
is isomorphic to the standard spectral triple
$\left( C^\infty\left(\mathbb{T}^2_{p/q}\right), L^2\left(\mathbb{T}^2_{p/q}\right)\otimes \mathbb{C}^2, D_{p/q}\right)$,
where $D_{p/q}$ is given by \eqref{Dirac} for $\theta=p/q$.
\end{prop}
Moreover we can equip $\left(\Gamma^\infty(F), L^2(F)\otimes \mathbb{C}^2, D_F\right)$ with a grading and real structure
and enhance the isomorphism
to an isomorphism of even real spectral triples.
\\
The suitable grading $\chi$ is just given by id$\otimes$diag(1,-1).
Furthermore it is evident that
%commutativity of diagram
\eqref{diagJ} holds for the following real structure:
  \begin{equation}\label{rsfirst}
  J_F=-i J^0_F \otimes (\sigma_2 \circ c.c.),
\end{equation}
 where $J^0_F$ acts on a section $\xi:\mathbb{T}^2\rightarrow F$ by hermitian conjugation, that is:
 \begin{equation}
 J^0_F\left(\sum_{m,n=1}^q f_{mn}{\xi_{U}}^{m}{\xi_{V}}^n \right)=\sum_{m,n=1}^q \overline{f}_{mn}\lambda^{-mn}{\xi_{U}}^{-m}{\xi_{V}}^{-n}.
\end{equation}
Notice that $J_F$ admits a decomposition along the infinite and finite dimensional component of its Hilbert space:
 \begin{equation}
  J_F=J \otimes h.c.
\end{equation}
where $J:L^2(\mathbb{T}^2,\Sigma)\otimes \mathbb{C}^2\rightarrow L^2(\mathbb{T}^2,\Sigma)\otimes \mathbb{C}^2$ is the charge conjugation on the spinor bundle of the commutative torus and $h.c.$ denotes fiber-wise hermitian conjugation on the matrix algebra $M_q.$
\subsection{Second isomorphic spectral triple}
\label{Second isomorphic spectral triple}

\noindent$\bullet$ The algebra $\mathcal{A}_{p/q}$.

Now we pass to another description of
$C^{\infty}\left(\mathbb{T}^2_{p/q}\right)$ which will explain its twined product nature.
The starting point is the natural bijective identification of
an arbitrary smooth section of the bundle $F$ with
a smooth function $\varphi:\mathbb{T}^2\to M_q$ that is
$\kappa$-$\rho$-equivariant, i.e.
$\varphi\circ\kappa_{m,n}= \rho_{m,n} \circ \varphi$,
or more explicitly
\begin{equation}\label{kapparho}
 \varphi(\lambda^m z_1,\lambda^n z_2)= R^mS^n
\varphi(z_1,z_2)S^{-n}R^{-m},
\end{equation}
via the algebra isomorphism
$$
\varphi \mapsto \xi_\varphi,
$$
where
$$
\xi_\varphi:\check{\mathbb{T}}^2\to F, \quad
[z_1,z_2]_\kappa\mapsto [z_1,z_2,\varphi(z_1,z_2)]_\tau\ .
$$

Next, we observe that
a smooth function $\varphi:\mathbb{T}^2\to M_q$ is
$\kappa$-$\rho$-equivariant as in \eqref{kapparho},
exactly when it is invariant under the pullback of the
$\tau$-action of $G$, which is
$$
{\tau}_{m,n}^*(\varphi):= \rho_{m,n} \circ \varphi_{m,n}\circ  \kappa_{m,n}.
$$
Furthermore, under the standard identification
\begin{equation}\label{stid}
C^{\infty}\left(\mathbb{T}^2,M_q\right)=
C^{\infty}\left(\mathbb{T}^2\right)\otimes M_q
\end{equation}
the subalgebra $C^{\infty}\left(\mathbb{T}^2,M_q\right)^\tau$
of ${\tau}^*$-invariant functions
corresponds to the subalgebra
$$\mathcal{A}_{p/q}:=\left(C^{\infty}(\mathbb{T}^2)\otimes M_q\right)^{\kappa\otimes\rho}$$
of invariant elements
under the action of the tensor product representation $\kappa\otimes\rho$ of $G$.

With these observations we can state:
\begin{lemma}\label{lemmaiso}
The map $T$ defined on the generators by
\begin{equation}\label{T2}
U\mapsto u:=  z_1\otimes S, \quad
V\mapsto v:=z_2\otimes R^{-1},
\end{equation}
where $z_\ell$ is the $\ell$-th coordinate function on $\mathbb{T}^2$,
extends to a $*$-isomorphism from the algebra $C^{\infty}\left(\mathbb{T}^2_{p/q}\right)$  to the algebra $\mathcal{A}_{p/q}=\left(C^{\infty}(\mathbb{T}^2)\otimes M_q\right)^{\kappa\otimes\rho}$.

\end{lemma}
\begin{proof}

By straightforward check using the properties of the Fourier coefficients of a smooth function and noting that $z_1\otimes S$ and $z_2\otimes R^{-1}$ are unitary and $\kappa\otimes\rho$-invariant, and satisfy the $\lambda$-exchange rule.

Next, it is easily seen that any element in $\mathcal{A}_{p/q}$
can be written as:\\
\begin{equation}\label{alt}
\sum_{r,s=0}^{q-1} f_{rs}(z_1,z_2)u^r v^s,
\end{equation}
where, for any $(r,s)\in (\mathbb{Z}/q)^2$,
$f_{rs}(z_1,z_2)$
are Schwarz functions on ${\mathbb{T}}^2$.
Since $u$ and $v$ are invariant such an element is
$\kappa\otimes\rho$-invariant if and only if
each $f_{rs}$ is $\kappa$-invariant, that is defines a function on $\check{\mathbb{T}}^2$.
Thus the set $\{u^mv^n\}$ for $m,n \in \mathbb{Z}$ is a basis of $\mathcal{A}_{p/q}$ over
$C^{\infty}\left(\check{\mathbb{T}}^2_{p/q}\right)$.
This shows surjectivity of $T$ and concludes the proof.

\end{proof}

\noindent
  Remark. Analogously to what has been observed for the first isomorphic spectral triple, also $Z(\mathcal{A}_{p/q})$ is isomorphic to $C^\infty(\check{\mathbb{T}}^2)$  and then  $C^\infty(\mathbb{T}^2_{p/q})$ and $\mathcal{A}_{p/q}$ are  isomorphic as left modules over  $C^\infty(\check{\mathbb{T}}^2)$ $\diamond $ \\

\noindent
$\bullet$ The Hilbert space $\mathcal{H}_{p/q}$.\\
\hspace*{6mm}
Let $\mathcal{H}^0_{p/q}:=\left(L^2(\mathbb{T}^2)\otimes M_q\right)^{\kappa\otimes\rho}$
be the Hilbert subspace of invariant elements in
$L^2(\mathbb{T}^2)\otimes M_q$
under the (extension of the bounded) action of the tensor product representation $\kappa\otimes\rho$ of $G$,
which also is the same as the obvious completion of $\mathcal{A}_{p/q}$.
For the spectral triple on $\mathcal{A}_{p/q}$ we take as Hilbert space
\begin{equation}
\mathcal{H}_{p/q}=\mathcal{H}^+_{p/q}\oplus \mathcal{H}^-_{p/q}=\mathcal{H}^0_{p/q}\otimes \mathbb{C}^2,
\end{equation}
where the apex $+$ and $-$ are just to mark which copy of $\mathcal{H}^0_{p/q}$ is in the $\pm1$ eigenspace of the grading operator $\gamma_{p/q}=$diag$(1,-1)$ Then,
taking advantage of the (inverse) isomorphism $T$ given by \eqref{T2} and its Hilbert space amplification, it is clear that:
 \begin{lemma}\label{im2}
The Hilbert modules
$ \left(\mathcal{A}_{p/q},\mathcal{H}_{p/q}\right)$
and $ \left( C^\infty(\mathbb{T}^2_{p/q}), L^2(\mathbb{T}^2_{p/q})\otimes \mathbb{C}^2\right)$ are ${*-{\rm isomorphic}}$.
\end{lemma}

\noindent$\bullet$ The Dirac operator $\mathcal{D}_{p/q}$.

Now we are going to select the Dirac operator $\mathcal{D}_{p/q}$ on $\mathcal{H}_{p/q}$ in such a way that
\eqref{diagD} is satisfied for $D_1=D_\theta$ and $D_2=\mathcal{D}_{p/q}$.
%the diagram \eqref{diagD} commute.
For sake of simplicity, we start with Weyl spinors
%, say left handed,
 (of grade +1) on the noncommutative torus.
By linearity, it is enough to check \eqref{diagD}
on each vector $u^mv^n= T(U^mV^n)$ of the basis, which in view of
\begin{equation}
\left(\delta_1 +i\delta_2\right)U^mV^n=i\left(m+in\right) U^mV^n
\end{equation}
%fixes $\mathcal{D}_{p/q}^+$ to be
requires that
\begin{equation}
\begin{split}
\mathcal{D}_{p/q}
u^mv^n\left(\begin{matrix} 1 \\  0\end{matrix}\right)
 & = i\left(m+in\right) u^m v^n \left(\begin{matrix} 0 \\  1\end{matrix}\right)
\\& =  i\left( m+in \right) \left(z^m\otimes R^m\right)\left(w^n\otimes S^n\right)\left(\begin{matrix} 0 \\  1\end{matrix}\right)
\\& = i\left( \partial_1 +i \partial_2 \right)\otimes 1 \left(z^m w^n\otimes R^m S^n\right)\left(\begin{matrix} 0 \\ 1\end{matrix}\right).
\end{split}
\end{equation}
Thus the appropriate Dirac operator on $\mathcal{H}_{p/q}$ reads
(modulo exchange of the tensor factors in $M_q\otimes \C^2$):
\begin{equation}\label{D2}
\mathcal{D}_{p/q}=\slashed{D}\otimes \mathbbm{1}_q,
\end{equation}
that has
%Notice that \eqref{D2} is of
the usual product form \eqref{product} though with a vanishing second term and acting not on the full tensor product of Hilbert spaces but only on its subspace.\\

\noindent$\bullet$ The isomorphism.\\
Similarly to the treatment of the first isomorphic spectral triple, by means of \eqref{lemmaiso} and \eqref{im2} and the above discussion, we obtain:
\begin{prop}
The spectral triple $\left(\mathcal{A}_{p/q},\mathcal{H}_{p/q},\mathcal{D}_{p/q}\right)$
is isomorphic to the standard spectral triple
$\left( C^\infty(\mathbb{T}^2_{p/q}), L^2(\mathbb{T}^2_{p/q})\otimes \mathbb{C}^2, D_{p/q}\right)$,
where $D_{p/q}$ is given by \eqref{Dirac} for a fractional $\theta=p/q$.
\end{prop}
Furthermore this isomorphism becomes an isomorphism of even real
spectral triples if we equip $\left(\mathcal{A}_{p/q},\mathcal{H}_{p/q},\mathcal{D}_{p/q}\right)$ with the grading
$\gamma_{p/q}$ as above and a real structure satisfying \eqref{diagJ} given by the $\C-$antiunitary operator
\begin{equation}\label{rssecond}
\mathcal{J}_{p/q}= -i\mathcal{J}^0_{p/q}\otimes (\sigma_2 \circ c.c.),
\end{equation}
where $\mathcal{J}^0_{p/q}$ acts by component-wise conjugation:
\begin{equation}
\mathcal{J}^0_{p/q}\left(f\otimes A\right)=\overline{f}\otimes A^*\ .
%, \quad \forall \; f\otimes A \in \mathcal{A}_{p/q}.
\end{equation}
%making the diagram \eqref{diagJ} commute.

\section{Inequivalent spin structures}\label{Inequivalent spin structures}
On the noncommutative torus the inequivalent spin structures correspond in a natural manner to double coverings.
This is most easily seen classically since $\T^2$
is parallelizable. Thus the structure group of its bundle of oriented orthonormal frames can be reduced to the trivial (one element) group,
and so the total space of such a reduced bundle is
just a copy of $\T^2$ itself, with a projection
on the base being the identity map.
Correspondingly also the whole spin structure can be reduced: the structure group $Spin(2)$ to its
two element subgroup $\Z_2$, the total space
 of the principal $Spin(2)$-bundle to a double cover of $\T^2$ and the spin structure map
 to the double covering map.
 The fully fledged (non reduced) spin structure can be reconstructed from the double cover as the bundle associated with the natural action of $\Z_2$ on $Spin(2)$ (as a subgroup).
It is also a matter of straightforward checking that
two such spin structures are equivalent precisely
iff the double covering are equivalent.

   Furthermore, in this reduced setting the Weyl spinors are just sections of the bundle associated with the faithful representation of $\Z_2$ on $\C$, or equivalently $\Z_2$-equivariant functions maps from the double cover to $\C$, or what is the same, (-1)-eigenfunctions of the generator of $\Z_2$. Then of course the Dirac spinors are just two copies of Weyl spinors.

All that makes sense also in the noncommutative realm by working dually in terms of algebras.
 The appropriate language is actually that of
 noncommutative double coverings, interpreted as
noncommutative principal $\Z_2$-bundles.
For our purposes this will essentially mean that we consider $C^*$ algebras that contain
$C^\infty\left(\mathbb{T}^2_{p/q}\right)$ as a subalgebra of index $2$.
More precisely we formulate it as follows.

\begin{mydef}\label{cover}
Let $A$ and $B$ be unital $C^*$ algebras. We say that $B$ is a noncommutative double covering of $A$ if $B$ is a graded algebra $B=B^0\oplus B^1$, such that 
%$B^0B^1=B^1$, 
the closure of $B^1B^1$ contains the unit of $B$, and $B^0$ is isomorphically identified with $A$.
Two double noncommutative coverings $B$ and $B'$ of $A$ are said to be equivalent iff there is a {$*-$isomorphism}
from  $B$ to $B'$ that restricts to identity on $A$.
\end{mydef}

This definition extends easily to suitable pre-$C^*$ algebras of $B$ and of $A$, and in particular to the case of smooth noncommutative torus.

\subsection{Inequivalent double coverings}\label{Inequivalent double coverings}
The  double coverings of noncommutative torus
 have been studied in \cite{D}
and just like in the classical case,
there exists four inequivalent of them,
and thus four inequivalent spin structures for arbitrary parameter $\theta$.
They are labelled by a pair of indices $j,k$
that take the values $0$ or $1$ and
the concrete algebras of the double covers  $C^\infty((\widetilde\Tb^2_{j,k})_\te)$
%that
are shown in the first column of the Table\,1 %\eqref{table}
below. The corresponding embedding homomorphisms
$$
h_{j,k}: C^\infty((\widetilde\Tb^2_{j,k})_\te)\hookleftarrow C^\infty(\Tb^2_\te)$$ of
$C^\infty(\Tb^2_\te)$ as subalgebras of index 2, send the generators
$U_{\te},V_{\te}\in C^\infty(\Tb^2_\te)$
to the elements listed respectively in the third and fourth column, where we also introduce a label on the generators
to indicate the parameter of the corresponding noncommutative torus.\\
\centerline {Table 1.}
\begin{equation}\label{table}
\begin{array}{|c||c|c|c|c|c|c|c|c|}
 \hline
% \vspace{1mm}
j, k                      &C^\infty\left((\widetilde\Tb^2_{j,k})_\te\right)&h_{j,k}(U_{\te})&h_{j,k}(V_{\te})\\ \hline\hline
0,0             &C^\infty\left(\Tb^2_\te\right) \otimes \bC^2
&U_{\te}\otimes{_1\choose^1}&V_{\te}\otimes{_1\choose^1}\\
\hline
1,0    &C^\infty\left(\Tb^2_\frac{\te}{2}\right)&U^{\,2}_{\frac{\!\te}{2}}&V_{\frac{\te}{2}}\\
\hline
0,1    &C^\infty\left(\Tb^2_\frac{\te}{2}\right)&U_{\frac{\te}{2}}&V^{\,2}_{\!\frac{\!\te}{2}}\\
\hline
1,1    & C^\infty\left(\Tb^2_\frac{\te}{4}\right)^{\bZ_2'}&
U_{\!\frac{\te}{4}}^{\,2} & V_{\frac{\!\te}{4}}^{\,2} \\  \hline
\end{array}
\end{equation}
In the last row the generator of $\bZ_2'$ acts by
$$
U_{\frac{\te}{4}}^m V_{\frac{\te}{4}}^n \mapsto
(-)^{m+n} U_{\frac{\te}{4}}^m V_{\frac{\te}{4}}^n
$$
and thus the generators of $\bZ_2'$-invariant subalgebra are $U_{\frac{\te}{4}}^m V_{\frac{\te}{4}}^n$ with $m+n$ even.

It is not difficult to see that the isomorphic images of $C^\infty(\Tb^2_\te)$
under the embeddings $
h_{j,k}$ are the subalgebras of $\bZ_2$-fixed elements
$$
C^\infty((\widetilde\Tb^2_{j,k})_\te)^{\bZ_2} = C^\infty(\Tb^2_\te)\ ,
$$
 where the generator of $\bZ_2$ acts by
$$
\begin{array}{c}
a\otimes {_w\choose^z}\mapsto a\otimes {_z\choose^w}, ~{\rm if}~ j,k=0,0 \ , \\
\\
%$$
%$$
U^m_{\frac{\te}{2}} V^n_{\frac{\te}{2}}\mapsto
(-)^m U^m_{\frac{\te}{2}} V^n_{\frac{\te}{2}},
~{\rm if}~ j,k=1,0 \ , \\
\\
%$$
%$$
U^m_{\frac{\te}{2}} V^n_{\frac{\te}{2}}\mapsto
(-)^n U^m_{\frac{\te}{2}} V^n_{\frac{\te}{2}},
~{\rm if}~ j,k=0,1 \ , \\
\\
%$$
%$$
U_{\frac{\te}{4}}^m V_{\frac{\te}{4}}^n \mapsto
(-)^{mn} U_{\frac{\te}{4}}^m V_{\frac{\te}{4}}^n,
~{\rm if}~ j,k=1,1\ ,
\end{array}
$$
where in the last case $m+n$ is even.
Note that although in the maximally twisted case the fourth root of
$\lambda$ is involved in $ C^\infty(\Tb^2_\frac{\te}{4})$,
only the square root of $\lambda$ really matters in $C^\infty(\Tb^2_\frac{\te}{4})^{\bZ_2'}$.
Anyhow a kind of `transmutation' occurs: the more twisted the spin structure is,
the more commutative parameter $\lambda$ is involved,
namely $\lambda$, $\lambda^{1/2}$, $\lambda^{1/4}$.\\

It is easy to find pairs of odd elements such that their product is the unit of $B$,
so that Definition\,\ref{cover} is satisfied by 
the four double coverings (so the spin structures) described above.
Furthermore we have:
%It turns out also that they are pairwise inequivalent.
% in the sense of Definition\,\ref{cover}:
\begin{prop}\label{inequiv}
The noncommutative double coverings $C^\infty((\widetilde\Tb^2_{j,k})_\te)$ are pairwise inequivalent.
\end{prop}
\begin{proof}
Consider the torus group $\Tb^2$ identified with the two-parameter group of automorphisms of $C^\infty(\Tb^2_\te)$
$$U^mV^n\mapsto w^mz^nU^mV^n, \quad (w,z)\in \Tb^2.$$
It is straightforward to check that the lifts of $\Tb^2$
as two-parameter group of automorphisms of $C^\infty((\widetilde\Tb^2_{j,k})_\te)$ form
four inequivalent usual double coverings $\widetilde\Tb^2_{j,k}$ of $\Tb^2$.
Namely, even though $\widetilde\Tb^2_{j,k}$ are all isomorphic as groups to $\Tb^2$, the covering maps are inequivalent for different value of $(j,k)$:
$$
\Tb^2\ni (w,z)\mapsto (w^j,z^k)\in \Tb^2.
$$
This can be seen most easily on the odd part
$C^\infty((\widetilde\Tb^2_{j,k})_\te)^1$ of the algebra $C^\infty((\widetilde\Tb^2_{j,k})_\te)$.
However, these double covers of $\Tb^2$
 would be equivalent if the noncommutative double coverings in question were equivalent, which concludes the proof.
\end{proof}

We we note in passing that the inequivalent double coverings of the group $\Tb^2$ regarded as lifts to inequivalent spin structures of the canonical action of $\Tb^2$ on $\Tb^2$  appeared already in \eqref{D0}.
%}

\subsection{Inequivalent spectral triples}\label{Inequivalent spectral triples}
Now we shall construct a spectral triple for each spin structure.
We simplify our notation as follows: given a covering $C^\infty((\widetilde\Tb^2_{j,k})_\te)$ we shall denote it as $\mathcal{C}_{j,k}$ and denote its even and odd part respectively as $\mathcal{C}^0_{j,k}$ and $\mathcal{C}^1_{j,k}.$
 In analogy to the commutative case, Dirac operators of these spin structures are obtained by lifting the Dirac operator on the base space $\mathbb{T}^2_\te$.\\

\noindent$\bullet$ Algebra.\\
 For every spectral triple as its algebra datum we take the even part $\mathcal{C}^0_{j,k}\approx
 C^\infty\left(\mathbb{T}^2_{\theta}\right)$

of the algebra of functions on the covering space.
  Instead, as discussed above, the smooth Dirac spinors are direct sum of smooth Weyl spinors, which are just
%anti-invariant
those elements
of $C^\infty((\widetilde\Tb^2_{j,k})_\te)$ that change sign under the action the generator of $\bZ_2$, that is
  the elements of the odd part $\mathcal{C}^1_{j,k}$.
 %  will play the role of smooth Weyl spinors.
 Therefore the space of smooth (two-component) Dirac spinors is just $\mathcal{C}^1_{j,k}\otimes \mathbb{C}^2.$\\

\noindent$\bullet$ Hilbert space.\\
Next to obtain a suitable Hilbert space we use the completion
$\overline{\mathcal{C}_{j,k}}$
of
$\mathcal{C}_{j,k}$ with respect to the norm
\begin{equation}\label{nss}
\mathfrak{t}\left(\sum_{(m,n)\in \mathbb{Z}^2}a_{mn}U^{m}V^{n}\right)=a_{00}
\end{equation}
and take as square integrable Weyl spinors the elements of its  odd part $\overline{\mathcal{C}^1_{j,k\,}}$ while as the
Hilbert space of Dirac spinors we take
$$\mathcal{H}_{j,k}:=\overline{\mathcal{C}^1_{j,k\,}}
\otimes \mathbb{C}^2.$$

%\red{
 %
\noindent$\bullet$ Dirac operator.\\
To construct the Dirac operator $D_{j,k}$ we start by lifting the derivations on the noncommutative torus to its double coverings, or more precisely, extending
the derivations  $\delta_\ell$, $\ell=1,2$ to derivations  $\tilde{\delta}_\ell$ of the algebras $\mathcal{C}_{j,k\,}$.
 In other words,
we  require the commutativity of the diagram:
 \begin{equation}\label{ssD}
   \xymatrix{\mathcal{C}_{j,k} \ar[r]^{\tilde{\delta}_\ell} & \mathcal{C}_{j,k} \\
C^\infty\left(\mathbb{T}^2_{\theta}\right)\ar[r]^{\delta_\ell}  \ar@{^{(}->}[u] &C^\infty\left(\mathbb{T}^2_{\theta}\right).\ar@{^{(}->}[u]}\\
 \end{equation}
 As easily seen $\tilde{\delta}_\ell$ are nothing but
the usual derivations on the noncommutative tori with the
modified parameters ${\theta},{\theta/2},{\theta/2}$ and ${\theta/4}$, respectively as in Table\,\eqref{table}, times
 a factor $1/2$ every time the $\ell$th entry of the spin structure label $j_k$ equals to $1$.
Then, the action of $\tilde{\delta}_\ell$ on $\mathcal{C}^1_{j,k\,}$
extends to unbounded densely defined operators on the completions $\overline{\mathcal{C}^1_{j,k\,}}$
and (diagonally in $\C^2$) on the Hilbert spaces %$L^2(\mathbb{T}^2_\te)$ to
$\mathcal{H}_{j,k}$.\\

\noindent
Remark.
It is not difficult to see that these operators
are precisely the infinitesimal generators of the lifted two-parameter groups of automorphisms
forming the four inequivalent double coverings of $\Tb^2$ as in the Proof of Proposition\,\ref{inequiv}. $\diamond$\\

Next we contract these operators (extended derivations) with the Pauli matrices (Clifford multiplication) to get an operator which acts on Dirac spinors:
 \begin{equation}
 D_{j,k}=i\left(\sigma_1 \tilde{\delta}_1 + \sigma_2 \tilde{\delta}_2\right).
\end{equation}
In the following table in the first column we list all $D_{j,k}$'s in terms of the usual derivations $\delta_1$ and $\delta_2$ defined on each covering noncommutative torus $\mathcal{C}_{j,k}$ (with parameters as in \eqref{table}), while in the second column we report their respective spectra $Spec\left(D_{j_,k}\right)$ as operators on $\mathcal{H}_{j,k}:$

\centerline {Table 2.}
  \begin{equation}
\begin{array}{|c||c|c|c|c|c|c|c|c|}
 \hline
 %\vspace{1mm}
j, k   &D_{jk}& Spec\left(D_{j_,k}\right)\\
\hline\hline
0,0     & i\sigma_1 \delta_1+  i\sigma_2 \delta_2  & \pm \sqrt{m^2+n^2} \\ \hline
1,0    & \frac{i}{2}\sigma_1 \delta_1+ i\sigma_2 \delta_2  & \pm \sqrt{\left(m+\frac{1}{2}\right)^2+n^2} \\ \hline
0,1    & i\sigma_1 \delta_1+  \frac{i}{2} \sigma_2 \delta_2  & \pm \sqrt{m^2+\left(n+\frac{1}{2}\right)^2} \\ \hline
1,1    &  \frac{i}{2}\sigma_1 \delta_1+  \frac{i}{2}\sigma_2 \delta_2  & \pm \sqrt{\left( m + \frac{1}{2} \right)^2+\left( n+\frac{1}{2} \right)^2}  \\  \hline
\end{array}
\end{equation}
\\

As it should, the spectral triple for the first spin structure agrees with the one given in Section \ref{The standard spectral triple}, since in that case both the even and odd subspaces (of functions and of Weyl spinors) are isomorphic with the algebra
$C^{\infty}(\mathbb{T}^2_\te)$.
Note also that for all the four inequivalent spin structures $D_{j_,k}$ are isospectral deformations (have the same spectra) of the classical case $\theta=0$.

\subsection{Isomorphic spectral triples}\label{Isomorphic spectral triples}

We now assume that $\theta=p/q$ and study isomorphic images of these spectral triples under both of the isomorphisms presented in sections \ref{First isomorphic spectral triple} and \ref{Second isomorphic spectral triple}.

As just mentioned, the case $(j,k)=(0,0)$ is identical to what we told in previous sections. Furthermore, for all the spin structures the algebra of smooth functions is isomorphic to $C^\infty\left(\mathbb{T}^2_\te\right)$ and hence both to $\Gamma(F)$ and $\mathcal{A}_{\frac{p}{q}}.$ Thus we shall just take care of the Hilbert spaces and Dirac operators, in the three nontrivial cases when
%starting with the case
$(j,k)\neq (0,0).$

In the following we let
$$
\begin{array}{c}
G^{10}=G^{01}=\mathbb{Z}_{2q}\times \mathbb{Z}_{2q}, \quad \lambda^{10}=\lambda^{01}=e^{2 \pi i \frac{p}{2q}},
\\
G^{11}=\mathbb{Z}_{4q}\times \mathbb{Z}_{4q}, \quad \lambda^{11}=e^{2 \pi i \frac{p}{4q}}\ .
\end{array}
$$
We denote by $\check{\mathbb{T}}^2$ the quotient of $\mathbb{T}^2$ by the free action $\kappa$ of $G^{jk}$, given by $\kappa_{m,n}(z_1,z_2)=\left(\left(\lambda^{jk}\right)^m z_1,\left(\lambda^{jk}\right)^ nz_2\right),$ similarly as in section \ref{First isomorphic spectral triple}.

Now let $(j,k)=(1,0)$ and take $R,S \in M_{2q}$ as in \eqref{mRS}and define an action $\tau$ of $G^{10}$ on $\mathbb{T}^2\times M_{2q}$ given by
$$\tau_{m,n}(z_1,z_2, A)=(\lambda^m z_1,\lambda^nz_2, R^mS^nAS^{-n}R^{-m}),$$ where $A\in M_{2q}$. We let $F'$ to be the orbit space of $\tau,$ that is a vector bundle over $\check{\mathbb{T}}^2$, with typical fibre $M_{2q}.$ The space of functions $\mathcal{C}^+_{1,0}$ is regarded as the $*-$subalgebra of $\Gamma(F')$ generated by $\xi_U^2$ and $\xi_V,$ where
$$
\xi_U:\check{\mathbb{T}}^2\to F', \quad
[z_1,z_2]_\kappa\mapsto [z_1,z_2,z_1S]_\tau\ ,
$$
$$
\xi_V:\check{\mathbb{T}}^2\to F', \quad
[z_1,z_2]_\kappa\mapsto [z_1,z_2,z_2R^{-1}]_\tau\ .
$$
The odd subalgebra  $\mathcal{C}^-_{1,0}$ is isomorphic to the linear span of ${U^{2m+1}V^n}_{(m,n)\in \mathbb{Z}^2}$ with coefficients $f_{mn}$ which are smooth functions on the torus. The Hilbert space of Weyl spinors  is isomorphic with its closure with respect to the norm defined by the scalar product
\begin{equation}
\left(g,f\right)
=\sum_{(m,n)\in\mathbb{Z}^2}\int_{\check{\mathbb{T}}^2}\overline{g}_{-m-n}f_{mn}.
\end{equation}
The Hilbert space $\mathcal{H}_{1,0}$ of Dirac spinors is as usual a direct sum of two Hilbert spaces of Weyl spinors.

Finally, the Dirac operator $D_{1,0}$ is unitarily equivalent to:
\begin{equation}
D_{1,0}^F=i(\sigma_1\partial_{1}^F  + \sigma_2\partial_{2}^F ),
\end{equation}
with $\partial_{1}^F$ and $\partial_{2}^F$ defined as in \eqref{df}.

Concerning the second isomorphic spectral triple, we have to regard $\mathcal{A}_{\frac{p}{q}}$ as the subalgebra of $\mathcal{A}_{\frac{p}{2q}}$ generated by $u_{1,0}=z_1^2 \otimes S$ and $v_{1,0}=z_2 \otimes R^{-1},$ while the Hilbert space $\mathcal{H}_{10}$ is isomorphic to its orthogonal complement, completed with respect to the scalar product defined by
\begin{equation}
\left(g,f\right)= \sum_{(m,n)\in\mathbb{Z}^2 }
\overline{g}_{-2m-n-1}f_{2m+n+1}\ .
\end{equation}
Then the Dirac operator $D_{10}$ is unitarily equivalent to the restriction of $
\mathcal{D}_{p/q}=\slashed{D}\otimes \mathbbm{1}_q
$ to $\mathcal{H}_{10}\otimes \mathbb{C}^2$ (modulo the exchange of the tensor factors as in \eqref{D2}).

Both these descriptions are similar for the spin structure $(0,1)$ provided that one exchanges the roles of $U$ and $V$, while for the fourth spin structure $(1,1)$ one has to replace every $M_{2q}$ with $M_{4q}$ and repeat the constructions above.\\

%\ver{
\noindent
A few remarks are in order.\\
%In section ???
We recalled after \cite{D} the four inequivalent spin structures as double coverings of
the noncommutative torus for any (in particular rational) parameter $\theta$
in the sense of noncommutative principal bundles.
What is a precise relation to the notions and classification in \cite{PS} and \cite{V13} is however unclear.
In particular note however that in contrast with \cite{PS}
it is claimed in \cite{V13} that some of the four
spin structures can be equivalent precisely in the case of rational parameter $\theta$.
%}

\section{Curved rational noncommutative torus}\label{Curved rational noncommutative torus}

So far we established isomorphisms between spectral triples by selecting Dirac operators, both for the first and second case, which act trivially as the identity on the finite part of the respective Hilbert spaces.
In order to push further the analogy with almost commutative manifolds, hence with standard model of particles (\cite{vS},\cite{koen}),
it would be interesting to get a Dirac operator whose action on the "internal" degrees of freedom is non trivial.
In quantum field theory, internal degrees of freedom of a single, isolated fermion change whenever it moves in a space-time region with a gauge field (e.g. electromagnetic) whose field strength is different from zero. From a more mathematical point of view (\cite{vS}\cite{MDV},\cite{DVKM}), having a non zero field strength corresponds to consider a Dirac operator which contains a non flat connection acting on the finite part of the Hilbert space. Roughly speaking, to have a Dirac operator which changes internal degrees of freedom of particles, one should consider the one that describes the analogue of a curved geometry of the internal directions.

Thus we are going to investigate some curved geometries on the rational noncommutative torus in view of the realizations of the spectral triple on $\mathbb{T}^2_{p/q}$ as twined almost commutative spectral triple.

The study of curvature on noncommutative torus was initiated in \cite{CT}, where  conformal rescaling
 of the the standard Dirac operator are considered, that have been later generalized to arbitrary conformal class in \cite{FK}.
A different perspective, that we are going to adopt, can be found in \cite{DS1} and \cite{DS2}, where the kind of perturbations employed preserves the boundedness of the commutator of the Dirac operators with any element of the algebra
%$C^\infty\left(\mathbb{T}^2_{p/q}\right),$
and relies on the real structure $J$ of the noncommutative torus.

We introduce a wide class of perturbations that include the rescalings with a conformal factor in $JC^\infty\left(\mathbb{T}^2_{\theta}\right)J^{-1}$ and the transformations studied in \cite{DS1} and \cite{DS2}
\begin{equation}\label{cDirac1}
 D^{(k)}= i\sum_{j,\ell=1,2}
 %^{2}
 \sigma^jk_j^\ell\,\delta_\ell\, k'^\ell_j + h.c.\ ,
 \end{equation}
 where $k^\ell_j, k'^\ell_j\in J C^{\infty}\left(\mathbb{T}^2_{p/q}\right)J^{-1}$ for
 $j,\ell\in \{1,2\}$ are assumed such that $D^{(k)}$ has compact resolvent.
 With this choice of $k^\ell_j, k'^\ell_j$ the operator $D^{(k)}$ maintains the bounded commutators with the algebra elements; a choice of $k^\ell_j, k'^\ell_j$ from the algebra $JAJ$ would lead to a spectral triple with twisted commutators.

It is expected that for rational  $\theta= p/q$ the counterparts of $D^{(k)}$
%as in \eqref{cDirac1}
on $L^2(F)\otimes \mathbb{C}^2$ and $\mathcal{H}_{p/q}$ will present some non trivial action on the finite part of the respective Hilbert spaces.
To see if this is the case we consider in the following two simple examples the special case
\begin{equation}\label{rightconf}
D^k=i\sigma^1\delta_1+i\sigma^2 k\delta_2, \quad {\rm where} \quad 0<k\in JAJ.
\end{equation}
\begin{exa}\label{exa1}
Let $k=J(U+U^*)J^{-1}+t$, where $t>2$. The action of $D^k$ on the basis of left handed Weyl spinors is given by
\begin{equation}
D^k\left(\begin{matrix}
U^mV^n
\\0
\end{matrix}\right)= \left(\begin{matrix}
0
\\ D^{k^+}(U^mV^n)
\end{matrix}\right)\ ,
\end{equation}
where
\begin{equation}
D^{k^+}(U^mV^n)=mU^mV^n+n(\lambda U^{m+1}V^n+\lambda^{-1} U^{m-1}V^n+t U^mV^n)\ ,
\end{equation}
whose isomorphic image under a similarity with the map $Q$ given by \eqref{map}
is given by
\begin{equation}
D^{{k}^+}_F (\xi_U^m\xi_V^n )=m\xi_U^m\xi_V^n+n(\lambda \xi_U^{m+1}\xi_V^n+\lambda^{-1} \xi_U^{m-1}\xi_V^n+t \xi_U^m\xi_V^n).
\end{equation}
Hence the operator $D^{k^+}$ is unitarily equivalent to
\begin{equation}
D^{{k}^+}_F=i\partial^F_1+A\, \partial^F_2:L^2\left(F\right)\rightarrow L^2\left(F\right)\ ,
\end{equation}
where $A\in \Gamma^\infty(End(F))$
has local components given by
\begin{equation}
\widetilde A({z}_1,{z}_2)=\lambda {z}_1 S
+\lambda^{-1}{z}_1^{-1}S^{-1}+t\mathbbm{1}_q\ .
\end{equation}
\end{exa}

From this Example we see that in general
the operators $D^{(k)}_F$ on the Hilbert space
$L^2(F)\otimes \mathbb{C}^2$ and
$\mathcal{D}^{(k)}_{p/q}$ on $\mathcal{H}_{p/q}$,
 do not admit the decomposition $D_1\otimes \mathbbm{1}_q+\chi_1\otimes D_2$ required for twined almost commutative manifolds. However this is not a surprise since even classically a product or twined product manifolds need not have such a structure on the metric level. Indeed in the noncommutative setting both for the description of the spectral triple on the associated vector bundle $F$ and for the description as  subalgebra of $C^\infty(\mathbb{T}^2)\otimes M_q,$  twined almost commutative spectral triples have been modelled on the invariant elements of the respective algebras under some action of $\mathbb{Z}_q \times \mathbb{Z}_q$ which intertwines the external and internal degrees of freedom. Thus, for these spectral triples, we are not able to write every admissible Dirac operator by simply joining together a Dirac operator on the standard spectral triple of the commutative torus with a Dirac operator on the finite space $M_q.$
Next we use an element $k$ in the center of $C^{\infty}\left(\mathbb{T}^2_{p/q}\right)$
to transform the Dirac operator.

\begin{exa}\label{exa2}
Let $k=U^q+U^{-q}+t$, where $t>2$. We consider again the action of $D^k$ on a basis of left-handed Weyl spinor as in the previous example
%\ref{exa1}
and similar computations lead to
\begin{equation}
D_F^{{k}^+}\xi_U^m\xi_V^n=m\xi_U^m\xi_V^n+n( \xi_U^{m+q}\xi_V^n+ \xi_U^{m-q}\xi_V^n+t \xi_U^m\xi_V^n).
\end{equation}

Hence we see that, with the same exchange of tensor factors as in \eqref{D2}, $D^k$ is unitarily equivalent to
\begin{equation}
{D}^k_F=
i{\sigma^1\partial}^F_1  + A\,\sigma^2{\partial}^F_2
%\otimes \mathbbm{1}_2
:L^2(F)\otimes\mathbb{C}^2\rightarrow L^2(F)\otimes\mathbb{C}^2\ ,
\end{equation}
where $A\in \Gamma^\infty(End(F\otimes \C^2))$
has local components given by
\begin{equation}
\widetilde A({z}_1,{z}_2)= z_1^q +{z}_1^{-q}+t\ .
\end{equation}
\end{exa}

This Example shows that the situation is different if we assume that in \eqref{cDirac1}
the elements   $k^\ell_j$ and $ k'^\ell_j$ belong to the center $\mathcal{Z}_{p/q}$ of $C^{\infty}\left(\mathbb{T}^2_{p/q}\right)$
(so in fact to the center of $JC^{\infty}\left(\mathbb{T}^2_{p/q}\right)J^{-1}$ too).
We prove the following proposition for transformations \eqref{cDirac1} of the isomorphic spectral triple $\left(\Gamma(F),L^2(F)\otimes \mathbb{C}^2,D_F\right)$, the proof for the corresponding transformations of the spectral triple  $\left(\mathcal{A}_{p/q},\mathcal{H}_{p/q},\mathcal{D}_{p/q}\right)$ is similar.

\begin{prop}Let $k^\ell_j,\,k'^\ell_j\in
 \mathcal{Z}_{p/q}$ and let $\check{k}^\ell_j,\,\check{k}'^\ell_j\in C^\infty({\mathbb{T}}^2)^{G}$ be the corresponding elements via the composed isomorphism
$\mathcal{Z}_{p/q}=C^\infty\!\left(\check{\mathbb{T}}^2\right)
=C^\infty({\mathbb{T}}^2)^{G}$
that sends ${U}^q\mapsto z_1^q$ and
${V}^q\mapsto z_2^q$.
The isomorphic image $D^{(k)}_F$ of the operator $D^{(k)}$ defined as in \eqref{cDirac1}
has a local expression of the twined almost commutative form with vanishing second term:
\begin{equation}
\widetilde{{D}_F^{(k)}}=\slashed{D}^{(\check{k})} %\otimes
\mathbbm{1}_q\ ,
\end{equation}
where
\begin{equation}\label{Dkk}
\slashed{D}^{(\check{k})}=
i\sum_{j,\ell=1}^{2}
 \sigma^j\check{k}_j^\ell\,\partial_\ell\, \check{k}'^\ell_j + h.c.
\end{equation}
is the accordingly transformed canonical Dirac operator on the classical torus $\mathbb{T}^2$.
\end{prop}

\begin{proof}
If ${k}_j^\ell$ and $k'^\ell_j$ are in the center of $C^{\infty}\left(\mathbb{T}^2_{p/q}\right),$ then their isomorphic images regarded as (scalar) multiplication operators on $L^2(F)\otimes \C^2$
have local expressions given by
$\check{k}_j^\ell$ and $\check{k}'^\ell_j$.
The statement then follows by \eqref{derFF}.
\end{proof}

This result can be seen as a consequence of the fact that, both for $\Gamma(F)$ and $\mathcal{A}_{p/q},$ the center of the algebra is isomorphic to $C^\infty(\mathbb{T}^2),$ thus a rescaling by an element in the center can only affect the part of Dirac operator which acts on functions of the classical torus.

\subsection*{Acknowledgements}
The second author was partially supported by
the grant H2020-MSCA-RISE-2015-691246-QUANTUM DYNAMICS.

%\newpage


\begin{thebibliography}{9}

\bibitem{vS}
W.\,van\,Suijlekom,\,Noncommutative\,Geometry\,and\,Particle\,Physics,\,Springer,\,$(2014).$

\bibitem{C1}
A. Connes Noncommutative geometry, Academic Press, $(1994).$
\bibitem{C2}
A. Connes, J. Lott, Particles models and noncommutative geometry, Nucl. Phys $B18$ Suppl., $29-47,$ $(1990).$
\bibitem{C4}
A. Chamseddine, A. Connes, The spectral action principle, Commun. Math. Phys. $186,$ $731-750,$ $(1997).$
arXiv:hep-th/9606001v1
\bibitem{CM}
A. Connes, M. Marcolli, Noncommutative geometry, quantum fields and motives, American Mathematical Society, $(2008).$
\bibitem{DLM}
A. Devastato, F. Lizzi, P. Martinetti, Grand Symmetry, Spectral Action, and the Higgs mass, JGEP $42,$ $(2014).$ arXiv:1403.7567v2
\bibitem{koen}
J. Boeijink,  K. van den Dungen, On globally non-trivial almost-commutative manifolds, J. Math. Phys. $55$, $103508,$ $(2014).$ arXiv:1405.5368v2
\bibitem{bondia}
J. Gracia-Bond\'{i}a M, J.C. Varilly, G. Figueroa, Elements of noncommutative geometry,
Birkhauser Advanced Texts, Basler Lehrbucher, $(2001).$
\bibitem{CDV}
A. Connes, M. Dubois-Violette \textit{"Noncommutative finite-dimensional manifolds. I. Spherical manifolds and related examples"}, $(2002).$ arXiv:math/0011194v3
\bibitem{CL}
A. Connes, G. Landi, Noncommutative Manifolds the Instanton Algebra and Isospectral Deformations, Commun.,Math.,Phys. $221,$ pp. $141-159,$ $(2001).$ arXiv:math/0011194v3
\bibitem{PS}
M. Paschke, A. Sitarz, On Spin Structures and Dirac Operators on the Noncommutative Torus,
Lett. Math. Phys.77, 317-327, $(2006).$ arXiv:math/0605191v1
\bibitem{V13}
J. J. Venselaar,
Classification of spin structures on the noncommutative n-torus
J. Noncommut. Geom. $7$,  $787-816,$ $(2013).$ arXiv:1003.5156v4
\bibitem{bcr16}
C. Bourne, C., A.L. Carey, A. Rennie,
A non-commutative framework for topological insulators, Rev. Math. Phys. $28 (2)$, 1650004-1-1650004-51, $(2016).$ arXiv:1509.07210v5
\bibitem{agi}
V. Aiello, D. Guido, T. Isola,
Spectral triples for noncommutative solenoidal spaces from self-coverings, J. Math. Anal. Appl. $448,$  $1378-1412,$ $(2017).$ arXiv:1604.08619
\bibitem{Canlubo}
C.R. Canlubo, Hopf algebroids, Hopf categories and their Galois theories,  $(2016).$ arXiv:1612.06317
\bibitem{M}
 B. Mesland, Spectral triples and KK-theory: A survey, Clay Mathematics Proceedings; Volume: $16$, Topics in Noncommutative Geometry, pp $197-212$  $(2012).$
arXiv:1304.3802
\bibitem{D0}
L. Dabrowski,
Group Actions on  Spinors. Lecture Notes. Bibliopolis, Naples, $(1988)$ (without proofs correction).
\bibitem{MDV}
 M. Dubois-Violette, Noncommutative differential geometry, quantum mechanics and gauge theory, In: Bartocci C., Bruzzo U., Cianci R. (eds) Differential Geometric Methods in Theoretical Physics, Lect. Notes Phys. $375$, Springer, $(1991).$
 \bibitem{DVKM}
 M. Dubois-Violette, R. Kerner, J. Madore, Gauge bosons in a non-commutative geometry, Phys. Lett. B $217$  $485-488,$ $(1989).$
\bibitem{FK}
 F. Fathizadeh, M. Khalkhali, The Gauss-Bonnet theorem for noncommutative two tori with a general conformal structure
J. Noncommut. Geom. $6,$ $(2010).$ arXiv:1110.3511v1
\bibitem{DS1}
L. D\k abrowski, A. Sitarz, Curved noncommutative torus and Gauss-Bonnet,
J. Math. Phys. $54$, 013518 (2013).
arXiv:$1204.0420v1,$
\bibitem{DS2}
L. D\k abrowski, A. Sitarz, An asymmetric nonccomutative torus,
%Symmetry, integrability and geometry: methods and applications
SIGMA $11,$ $(2015).$  arXiv:1406.4645v3
\bibitem{D}
L. D\k abrowski, Spinors and theta deformation, Russ. J. Math.
Phys. $16$, $404-408,$ $(2009).$  arXiv:0808.0440v1
\bibitem{M}
 B. Mesland, Spectral triples and KK-theory: A survey, Clay Mathematics Proceedings; Volume: $16$, Topics in Noncommutative Geometry, pp $197-212$  $(2012).$
arXiv:1304.3802
 \bibitem{CT}
A. Connes, P. Tretkoff, The Gauss-Bonnet theorem for the noncommutative torus, $(2009).$ arXiv:0910.0188v1


\end{thebibliography}
\end{document}